\documentclass{amsart}
\usepackage{amssymb}
\usepackage{amsmath}
\usepackage{mathrsfs}
\usepackage{graphicx}
\usepackage[all]{xy}

\title{The ring of projective invariants of eight points on the line via
representation theory}

\author[Benjamin Howard, John Millson, Andrew Snowden and Ravi Vakil]
{Benjamin Howard, John Millson, Andrew Snowden and Ravi Vakil*}

\thanks{*B. Howard was partially supported by NSF grants DMS-0405606 and
DMS-0703674.  J. Millson was partially supported by the Simons Foundation
and NSF grants DMS-0405606 and DMS-05544254.  R. Vakil was partially supported
by NSF grant DMS-0801196.}

\date{September 7, 2008.}

\newtheorem{theorem}{Theorem}[section]
\newtheorem{lemma}[theorem]{Lemma}
\newtheorem{corollary}[theorem]{Corollary}
\newtheorem{proposition}[theorem]{Proposition}

\newtheorem{problem}[theorem]{Problem}

\theoremstyle{definition}

\theoremstyle{remark}
\newtheorem{remark}[theorem]{\it Remark}

\newcommand{\sfrac}[2]{{\textstyle \frac{#1}{#2}}}

\newcommand{\cut}[1]{}

\DeclareMathOperator{\Proj}{Proj}
\DeclareMathOperator{\Sym}{Sym}

\DeclareMathOperator{\bw}{{\bigwedge}^{\! 2}}

\DeclareMathOperator{\Hom}{Hom}
\DeclareMathOperator{\Aut}{Aut}
\DeclareMathOperator{\End}{End}
\DeclareMathOperator{\sgn}{sgn}
\DeclareMathOperator{\Tor}{Tor}
\DeclareMathOperator{\Spec}{Spec}
\DeclareMathOperator{\rk}{rk}
\DeclareMathOperator{\chr}{char}

\def\GL{\mathrm{GL}}
\def\SL{\mathrm{SL}}
\def\Sp{\mathrm{Sp}}

\def\Z{\mathbb{Z}}

\def\P{\mathbb{P}}
\def\C{\mathbb{C}}
\def\Q{\mathbb{Q}}
\def\F{\mathbb{F}}

\def\etc{{\it etc.}}

\def\half{\sfrac{1}{2}}

\def\wt{\mathbf{wt}}
\def\cq{/\!/}

\let\ms\mathscr
\let\mf\mathfrak
\let\mc\mathcal
\let\ss\scriptstyle
\let\wt\widetilde
\let\ol\overline

\begin{document}


\begin{abstract}
  The ring of projective invariants of eight ordered points on the
  line is a quotient of the polynomial ring on $V$, where $V$ is a
  fourteen-dimensional representation of $S_8$, by an ideal $I_8$,
  so the modular fivefold $(\P^1)^8 \cq\, \GL(2)$ is $\Proj(\Sym^{\bullet}(V)/I_8)$.
  We show that there is a unique cubic hypersurface $S$ in $\P V$ whose
  equation $s$ is skew-invariant, and that the singular locus of $S$
  is the modular fivefold.  In particular, over $\Z[1/3]$, the modular
  fivefold is cut out by the $14$ partial derivatives of $s$.  Better:
  these equations generate $I_8$.  In characteristic 3, the cubic $s$ is
  needed to generate the ideal.  The existence of such a cubic was
  predicted by Dolgachev. Over $\Q$, we recover the $14$ quadrics
  found by computer calculation by Koike \cite{koike}, and our
  approach yields a conceptual representation-theoretic description of
  the presentation.  Additionally we find the graded Betti numbers 
  of a minimal free resolution in any characteristic.

  The proof over $\Q$ is by pure thought, using Lie theory and commutative algebra.
  Over $\Z$, the assistance of a computer
  was necessary.  This result will be used as the base case
  describing the equations of the moduli space of an arbitrary number of points
  on $\P^1$, with arbitrary weighting, in \cite{HMSV}, completing
  the program of  \cite{hmsv1}.  The modular fivefold, and corresponding ring,
  are known to have a number of
  special incarnations, due to Deligne-Mostow, Kondo, and Freitag-Salvati
  Manni, for example as ball quotients or ring of modular forms respectively.
\end{abstract}

\maketitle

\tableofcontents

\section{Introduction}

Let $n$ be an even integer, let $M_n=(\P^1)^n \cq \GL(2)$ be the GIT
quotient of $n$ ordered points on the projective line, and let $R_n$
be the projective coordinate ring of $M_n$.

The first general results on the ring $R_n$ were found by Kempe \cite{Kempe}
in 1894.  He showed that elements of $R_n$ can naturally be interpreted
as (formal linear combinations of) regular graphs on $n$ vertices.  We
review this theory in \S \ref{s-review}. 
Kempe used this insight to  prove that $R_n$ is generated in degree one.
Thus $R_n$ has a natural set of generators: the \emph{matchings}
(regular degree one graphs) on
$n$ vertices.

Let $I_n$ be the {\em ideal of relations}, the kernel of
$\Sym(R_n^{(1)}) \to R_n$.  The problem of giving a presentation for
$R_n$ is thus reduced to determining generators for the ideal $I_n$.
For $n=4$, $I_n=0$.  For $n=6$, the ideal is generated by a single,
beautiful cubic equation:  the Segre cubic relation (see for example
\cite[p.~17]{DO}).  In \cite{HMSV}, we will show that for $n \geq 8$,
away from small characteristic, $I_n$ is generated by an explicit
simple class of quadrics.  The argument will rely
on the base case $n=8$, the subject of this paper.  It turns
out that this case (like the cases $n=4$ and $n=6$) has some special
extrinsic geometry.

Our main theorems are the following.  Note that $R_8^{(1)}$ is 14
dimensional.

\begin{theorem}
\label{mainthm}
Over $\Q$, there is a unique (up to scaling) non-zero skew-invariant cubic
polynomial $s$ in $\Sym^3 ( R_8^{(1)})$.  It vanishes on $M_8$, and $M_8$ is
the singular locus of $s=0$.  Better: the 14 partial derivatives of $s$
generate $I_8$ --- the singular scheme of the affine cone of $s=0$ is
precisely the affine cone over $M_8$.  These 14 partial derivatives
have no syzygies of degree zero or one.  In terms of graphs, $s$ may
be taken to be the skew-average of the cube of any matching.
\end{theorem}

This result was predicted to us by Igor Dolgachev.  It will be proved by pure
thought, that is, without the use of a computer or long, explicit formulas.
One consequence is a natural duality between the degree~1 piece of the ring
(with representation corresponding to the partition $4+4$) and the degree~2
piece of the ideal (with sign-dual representation $2+2+2+2$) into the sign
representation given by the cubic.

With the aid of a computer, we have a stronger integrality result:

\begin{theorem}
\label{mainthm2}
Over $\Z$, there is a non-zero cubic polynomial $s'$ in $\Sym^3(
R_8^{(1)})$ (a rational multiple of the $s$ of Theorem~\ref{mainthm})
such that $M_8$ is the singular locus of $s'=0$.  Better: the
ideal $I_8$ is generated over $\Z$ by $s'$ and its 14 partial
derivatives.  In particular, $I_8$ is generated over $\Z[1/3]$ by the
14 partial derivatives of $s'$.  Over $\Z$, the cubic $s'$
is not generated by its partial derivatives.
\end{theorem}

This result is proved and discussed further in \S \ref{s-integrality}.

\begin{remark}
It would be ideal, of course, to have a pure thought proof of
Theorem~\ref{mainthm2}.  Here is where we stand with respect to this.
Theorem~\ref{mainthm} automatically holds over $\Z[1/N]$ for some integer $N$.
This integer cannot be determined from our proofs.  However, using analogues
of our proofs in positive characteristic, one can obtain a precise value of
$N$ so that the theorem remains true over $\Z[1/N]$.  Unfortunately, we
cannot get down to $N=3$ by these methods.  Since the positive characteristic
arguments are more complicated, use less well-known facts and do not yield an
optimal result, we decided not to include them.
\end{remark}

\subsection{Other manifestations of this space, and this graded ring}
\label{s:1-1}

The extrinsic and intrinsic geometry of $M_n$ for small $n$ has
special meaning often related to the representation theory of $S_n$.
For example, $M_4$ relates to the cross ratio, $M_5$ is the quintic
del Pezzo surface, and the geometry of the Segre cubic $M_6$ is well
known (see for example \cite{hmsv6} for the representation theory).
The space $M_8$ might be the last of the $M_n$ with such individual
personality.  For example, over $\C$, the space may be interpreted as
a ball quotient in two ways:
\begin{enumerate}
\item Deligne and Mostow \cite{dm} showed that $M_8$ is isomorphic to
  the Satake-Baily-Borel compactification of an arithmetic quotient of
  the $5$-dimensional complex ball, using the theory of periods of a
  family of curves that are fourfold cyclic covers of $\P^1$
  branched at the $8$ points.
\item Kondo \cite{kondo} showed that $M_8$ may also be interpreted in
  terms of moduli of certain K3 surfaces, and thus $M_8$ is isomorphic
  to the Satake-Baily-Borel compactification of a quotient of the
  complex $5$-ball by $\Gamma(1-i)$, an arithmetic subgroup of a
  unitary group of a hermitian form of signature $(1,5)$ defined over
  the Gaussian integers.  See also \cite[p.~12]{fs} for further
  clarification and discussion.
\end{enumerate}
Both interpretations are $S_8$-equivariant (see  \cite[p.~8]{kondo}
for the second).

Similarly, the graded ring $R_8$ we study has a number of manifestations:
\begin{enumerate}
\item It is the ring of genus $3$ hyperelliptic modular forms of level 2.
\item Freitag and Salvati Manni showed that $R_8$ is isomorphic to the
  full ring of modular forms of $\Gamma(1-i)$ \cite[p.~2]{fs}, via
 the Borcherds additive lifting.
\item The space of sections of multiples of a certain line bundle on
$\ol{\mc{M}}_{0,8}$ (as there is a morphism $\ol{\mc{M}}_{0,8} \to M_8$,
\cite{kapranov}, see also \cite{avritzerlange}).
\item Igusa \cite{igusa} showed that there is a natural map $A(\Gamma_3[2])/
\mc{I}_3[2]^0 \to R_8$, where $A(\Gamma_3[2])$ is the ring of Siegel modular forms
of weight 2 and genus 3. (See \cite[\S 3]{fs} for more discussion.)\item It is a quotient of the third in a sequence of algebras related to the
orthogonal group $\mathrm{O}(2m, \F_2)$ defined by Freitag and Salvati
Manni (see \cite{fs1}, \cite[\S 2]{fs}).  (The cases $m=5$ and $m=6$ are
related to Enriques surfaces.)
\end{enumerate}
The Hilbert function $f(k) = \dim{R_8^{(k)}}$ was found by Howe
\cite[p.~155, \S 5.4.2.3]{howe}:
\begin{displaymath}
f(k) = \sfrac{1}{3} k^5 + \sfrac{5}{3} k^4 + \sfrac{11}{3} k^3+
\sfrac{13}{3} k^2+3k+1,\qquad \textrm{for $k \ge 0$.} 
\end{displaymath}
The Hilbert series $H(t) = \sum_{k=0}^\infty f(k) t^k$ is 
\begin{displaymath}
H(t) = \frac{1+8t+22t^2+8t^3+t^4}{(1-t)^6}.
\end{displaymath}
Both of these formulas are given in \cite[p.~7]{fs}.

One reason for $M_8$ to be special is the coincidence $S_8 \cong
\mathrm{O}(6, \F_2)$.  A geometric description of this isomorphism in this
context is given in \cite[\S 4]{fs}.  Another reason is Deligne and Mostow's
table \cite[p.~86]{dm}.

\subsection{Other manifestations of the cubic}

Let $n$ be an even integer.
There are natural generators of $R^{(1)}_n$, one for each directed matching
on $n$ labeled vertices (see \S \ref{s-graph}).  The group $S_n$ acts on
these coordinates in the obvious way.  The signed sum of the cubes of these
matchings $s_n$, regarded as an element of $\Sym^3(R_n^{(1)})$, is
skew-invariant.  By skew-invariance, it must vanish on those points of
$M_n$ where two of the $n$ points come together.  Hence $s_n$ must be
divisible by the discriminant, which has degree $\half \binom{n}{2}$.  Thus for
$n \geq 6$, $s_n$ vanishes on $M_n$.  (This cubic appeared in the e-print
\cite[\S 2.10]{hmsveprint}, but was removed in the published version because
its centrality was not yet understood.)  For $n=4$, $s_4$ vanishes precisely
on the boundary of $M_4$.  For $n=6$, $s_6$ is the Segre cubic.  For $n=8$,
$s_8$ is the $s$ of Theorem~\ref{mainthm} (although it must be scaled to give
the $s'$ of Theorem~\ref{mainthm2}).  And for $n > 8$, it may be shown that
$s_n=0$, so $n=8$ is indeed the last interesting case.
\label{other}

\subsection{Outline of the proof of Theorem~\ref{mainthm}}

We now describe the main steps in the proof of Theorem~\ref{mainthm}:
\begin{enumerate}
\item We first prove the existence and uniqueness of the cubic $s$, using
just linear algebra.
\item Next we prove that the partial derivatives of $s$ generate $I_8^{(2)}$,
using the structure of
the relevant spaces as $S_8$-modules.
\item We then prove that the partial derivates of $s$ have no linear syzygies.
This is where most of the work occurs.
\item Finally we use the fact that $R_8$ is Gorenstein and step (3) to fill in the
Betti diagram of $R_8$.  From this we see that $I_8$ is generated by quadrics.
\end{enumerate}
Steps (3) and (4) can of course be replaced by Koike's computer calculation
\cite{koike}, at the expense of the conceptual argument.  As a simple corollary
of the step (3) we find that $I_8^{(2)}$ generates $I_8^{(3)}$.  We could then
replace step (4) by an appeal to a result in \cite{HMSV} which states that $I_n$
is generated by $I_n^{(2)}$ and $I_n^{(3)}$ for any $n$.  However, we prefer to
avoid referring to a later paper.

Step (3) may be further broken down as follows:
\begin{enumerate}
\item[(3a)] Let $\Psi:\End(R_8^{(1)}) \to I_8^{(3)}$ be the map given by $A
\mapsto A s$, where $As$ is defined via the natural action of the Lie algebra
$\End(R_8^{(1)}) \cong \mf{gl}(14)$ on $\Sym^3(R_8^{(1)})$.  We first observe that
the space of linear syzygies between the partial derivaties of $s$ is exactly
$\mf{g}=\ker{\Psi}$.  We note that $\mf{g}$ is a Lie subalgebra of
$\End(R_8^{(1)})$ and is stable under the action of $S_8$.
\item[(3b)] Next, using general theory developed in \S \ref{s-liesub}
concerning $G$-stable Lie subalgebras of $\End(V)$, where $V$ is a
representation of $G$, and the classification of simple Lie algebras, we show
that the only $S_8$-stable Lie subalgebras of $\End(R_8^{(1)})$ are 0,
$\mf{so}(14)$ and $\mf{sl}(14)$ (ignoring the center).  Thus $\mf{g}$ must be
one of these three Lie algebras.
\item[(3c)] Finally, we show that $\mf{so}(14)$ does not annihilate any
non-zero cubic.  As $\mf{g}$ is the annihilator of $s$ we conclude $\mf{g}=0$.
\end{enumerate}

\subsection{Relationship with \cite{HMSV}}

We now discuss the relationship between this paper and \cite{HMSV}.  The two
papers prove similar results but are logically and methodologically independent.
(Except for a few peripheral remarks in this paper that rely on \cite{HMSV}.)
In \cite{HMSV} we prove that $I_n$ is generated by quadrics for $n \ge 8$.  The
argument is inductive and uses the $n=8$ case for its base.  This base case is
already known by the work of Koike and so, strictly speaking, \cite{HMSV} does not
logically rely on the present paper.  However, Koike's proof of the $n=8$ case was
by a computer calculation.  Thus the present paper can be viewed as filling this
conceptual gap.  Together, this paper and \cite{HMSV} give a complete conceptual
proof that $I_n$ is generated by quadrics for $n \ge 8$.

As the main results of this paper and \cite{HMSV} are both concerned with
quadric generation of $I_n$ one might think that it would make more sense to
combine the two papers.  We feel that this is not the case for two reasons.
First, as we have highlighted above, the $n=8$ case has a number of special
properties not shared by the general case.  Had we combined the two papers, we
feel that these beautiful features would have been obscured in the resulting, much
larger paper.  And second, although the results of the two papers are similar,
the methods of proof are completely different.  This paper uses Lie theory and
commutative algebra while the main tools of \cite{HMSV} are toric degenerations
and combinatorics.

\subsection{Other results}

In the course of our study we have found some miscellaneous results which
do not fit into the rest of the paper.

First, Miles Reid pointed out that the secant variety of $M_8$ necessarily
lies in the cubic $s=0$, by Bezout's theorem.  The secant variety is
$11$-dimensional, as expected (this is a computer calculation), and is thus
a hypersurface in the cubic.

The second result concerns the ring $R_8$.  As stated above, elements of $R_n$
may be represented as formal sums of regular graphs on $n$ vertices.  In
particular, the matchings on $n$ points span $R_n^{(1)}$.  By embedding
the $n$ vertices into the unit circle we obtain a notion of planarity.
A theorem of Kempe states that the planar graphs give a basis for $R_n$.  We
observed (using a computer) the following result, which is particular to
the case of 8 points:

\begin{proposition}
The squares of the non-planar matchings form a basis for $R_8^{(2)}$.  This
holds over $\Z[1/2]$.
\end{proposition}

\subsection{Acknowledgments}
Foremost we thank Igor Dolgachev, who predicted to us that
Theorem~\ref{mainthm} is true.  Without him this paper would not have
been written.  We also thank Shrawan Kumar and Riccardo Salvati Manni
for helpful comments.

\section{Review of the ring $R_L$}
\label{s-review}

In this section we give a precise definition of the ring $R_n$ and recall
some facts about how $S_n$ acts on $R_n$.  A more thorough treatment of these
topics is given in \cite{HMSV}.

Before we begin, we remark that we prefer to work as functorially as possible.
This results in  greater clarity and does not cost
much.  Thus, rather than working with an integer $n$ we work with a set $L$
of cardinality $n$.  We will therefore have a ring $R_L$ in place of $R_n$.
Also, rather than working with $\P^1$ we work with $\P{U}$ where $U$ is a
two-dimensional vector space.  For this section we work over an arbitrary
commutative base ring $k$.  In most of the remainder of the paper we will
take $k$ to be a field of characteristic $0$.

\subsection{The ring $R_L$}

Let $k$ be a commutative ring, let $L$ be a finite set of even cardinality
$n$, let $U$ be a free rank two $k$-module and let $\omega$ be a
non-degenerate symplectic form on $U$.  We are interested in the GIT quotient
$M_L=(\P{U})^L \cq \GL(U)$.
By definition, $M_L$ is $\Proj(R_L)$ where
\begin{displaymath}
R_L= \bigoplus_{d=0}^{\infty} \Gamma( (\P{U})^L, \ms{O}(d)^{\boxtimes
L})^{\GL(U)}.
\end{displaymath}
Here the action of $\GL(U)$ on the $d$th-graded piece is the usual action
twisted by the $(-d/2)$th power of the determinant.  Thus taking
$\GL(U)$-invariants is the same as taking $\SL(U)$-invariants.  We can
therefore rewrite the above formula as
\begin{equation}
\label{eq-defr}
R_L=\left[ \bigoplus_{d=0}^{\infty} (\Sym^d{U^*})^{\otimes{L}}
\right]^{\SL(U)}.
\end{equation}
Here $U^*=\Hom(U, k)$.  We take \eqref{eq-defr} as the \emph{definition} of
$R_L$ for the purposes of this paper.  Note that $R_L$ and $M_L$ depend upon
$k$ but this is absent from the notation.  We will write $(R_L)_k$ when we
want to emphasize the dependence on $k$.  Of course, $(R_L)_k=(R_L)_{\Z}
\otimes k$. We
note that the symmetric group $S_L=\Aut(L)$ acts on $R_L$ by permuting the
tensor factors.

We now take a moment to comment about tensor powers, such as the one appearing
in \eqref{eq-defr}.  For a $k$-module $V$ and a finite set $L$ we define
$V^{\otimes L}$ in the obvious manner:  it is the universal $k$-module
with a multilinear map from $\Hom(L, V)$.  We think of pure tensors in
$V^{\otimes L}$ as \emph{functions} from $L$ to $V$.  For an integer $n$,
we write $V^{\otimes n}$ for $V^{\otimes \{1, \ldots, n\}}$.  The
construction $V^{\otimes L}$ is functorial in both $V$ and $L$.

\subsection{Graphical description of $R_L$}
\label{s-graph}

Let $e=(i, j)$ be an element of $L \times L$.  Since $e$ is ordered, we have
a natural isomorphism $(U^*)^{\otimes \{i, j\}}=(U^*)^{\otimes 2}$.  We can
thus transfer the symplectic form $\omega$ on $U$, thought of as an element of
$(U^*)^{\otimes 2}$, to an element $\omega_e$ of $(U^*)^{\otimes \{i, j\}}$.
Explicitly, if we pick a symplectic basis $\{x, y\}$ of $U^*$ then $\omega_e=
x_i y_j - x_j y_i$, where $x_i y_j$ is just shorthand for the function $\{i,
j\} \to U^*$ which takes $i$ to $x$ and $j$ to $y$.  Note that $\omega_e$ is
invariant under $\SL(U)=\Sp(U)$.

We now give a description of $R_L$ in terms of graphs.  We say that a
directed graph with vertex set $L$ is \emph{regular} if each vertex has the
same valence.  This common valence is then called the \emph{degree} of the
graph.  Let $\Gamma$ be a regular directed graph of degree $d$.  We define an
element $X_{\Gamma}$ of $R_L^{(d)}$ by $X_{\Gamma}=\prod \omega_e$, the
product taken over the edges $e$ of $\Gamma$.
(These may be interpreted as Specht polynomials.)
As each $\omega_e$ is invariant
under $\SL(U)$, so is $X_{\Gamma}$.  The fact that $\Gamma$ is regular of
degree $d$ ensures that $X_{\Gamma}$ belongs to $(\Sym^d{U^*})^{\otimes L}$.

It is a fact from classical invariant theory that the $X_{\Gamma}$ span
$R_L$ as a $k$-module.  The next matter, of course, is to determine the
relations between the various $X_{\Gamma}$.  To begin with, we clearly have
$X_{\Gamma} X_{\Gamma'}=X_{\Gamma \cdot \Gamma'}$, where $\Gamma \cdot \Gamma'$
denotes the graph on $L$ whose edge set is the union of those of $\Gamma$ and
$\Gamma'$.  We then have the following easily verified relations:
\begin{itemize}
\item (Sign relation.) $X_{\Gamma}=-X_{\Gamma'}$ if $\Gamma'$ is
obtained from $\Gamma$ by reversing the direction of a single edge.
\item (Loop relation.) $X_{\Gamma}=0$ if $\Gamma$ contains a loop.
\item (Pl\"ucker relation, see Fig.~\ref{fig:pluck}.)  Let $\Gamma$ be a
regular directed graph and let $(a, b)$ and $(c, d)$ be two edges of $\Gamma$.
Let $\Gamma'$ (resp.\ $\Gamma''$) be the graph obtained by replacing these two
edges with the edges $(a, d)$ and $(c, b)$ (resp.\ $(a, c)$ and $(b, d)$).
Then
\begin{displaymath}
X_{\Gamma}=X_{\Gamma'}+X_{\Gamma''}.
\end{displaymath}
\end{itemize}
It is now a second fact from classical invariant theory that these three
types of relations generate all the relations amongst the $X_{\Gamma}$.  (The
loop relation is implied by the sign relation if 2 is invertible in $k$.)

\begin{figure}
\begin{displaymath}
\begin{xy}
(0, -6)*{\ss a}; (0, 6)*{\ss d}; (8, -6)*{\ss c}; (8, 6)*{\ss b};
(0, -4)*{}="A"; (0, 4)*{}="B"; (8, -4)*{}="C"; (8, 4)*{}="D";
{\ar "A"; "D"}; {\ar "C"; "B"};
\end{xy}
\qquad = \qquad
\begin{xy}
(0, -6)*{\ss a}; (0, 6)*{\ss d}; (8, -6)*{\ss c}; (8, 6)*{\ss b};
(0, -4)*{}="A"; (0, 4)*{}="B"; (8, -4)*{}="C"; (8, 4)*{}="D";
{\ar "A"; "B"}; {\ar "C"; "D"};
\end{xy}
\qquad + \qquad
\begin{xy}
(0, -6)*{\ss a}; (0, 6)*{\ss d}; (8, -6)*{\ss c}; (8, 6)*{\ss b};
(0, -4)*{}="A"; (0, 4)*{}="B"; (8, -4)*{}="C"; (8, 4)*{}="D";
{\ar "A"; "C"}; {\ar "D"; "B"};
\end{xy}
\end{displaymath}
\caption{The Pl\"ucker relation.}
\label{fig:pluck}
\end{figure}
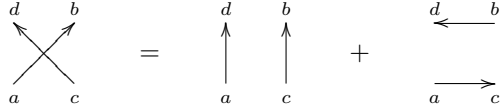

To be a bit more precise, let $\wt{R}_L$ be the free graded $k$-module with
basis $\{ \wt{X}_{\Gamma} \}$ as $\Gamma$ varies over directed regular graphs
on $L$.  The grade of $\wt{X}_{\Gamma}$ is the degree of $\Gamma$.  We turn
$\wt{R}_L$ into a ring by defining $\wt{X}_{\Gamma} \wt{X}_{\Gamma'}=
\wt{X}_{\Gamma \cdot \Gamma'}$.  We then have a map $\wt{R}_L \to R_L$ given
by $\wt{X}_L \mapsto X_L$.  The two theorems of classical invariant theory
referred to above about amount to the assertion that this map is surjective
and the kernel is generated by the sign, loop and Pl\"ucker relations.

\subsection{Facts needed about $R_L$}

We now recall some of its properties of $R_L$ that will be
relevant to us.  The first and perhaps most important is the
following (for a proof, see \cite{HMSV} or  \cite{hmsv1}):

\begin{proposition}[Kempe]
The ring $R_L$ is generated as a $k$-algebra by its degree one piece.
\end{proposition}

We emphasize that this holds for all $k$, or equivalently, for $k=\Z$.
We remark that Kempe proved another theorem:  if one fixes an embedding of
$L$ into the unit circle, so that one can make sense of what it means for
a graph to be planar, then the $X_{\Gamma}$ with $\Gamma$ planar form an
basis of $R_L$ as a $k$-module.  Thus, for instance, one can count planar
graphs to determine the Hilbert function of $R_L$.

We write $V_L$ for the first graded piece of $R_L$.  It is spanned by regular
graphs of degree one.  We call such graphs \emph{matchings}.  Kempe's theorem
says that the map $\Sym(V_L) \to R_L$ is surjective.  We let $I_L$ be the
kernel of this map.  We call $I_L$ the \emph{ideal of relations}.  The present
paper is concerned with finding generators for $I_L$ when $L$ has cardinality
eight.

We need to recall some facts about how $V_L$ and some related spaces decompose
under the symmetric group $S_L =\Aut(L)$.  For simplicity, we now take $k$ to be a
field of characteristic zero, although analogues of the statements remain true so
long as $n!$ is invertible in $k$.  Recall that the irreducible representations of
$S_L$ over $k$ correspond to Young diagrams, or partitions.  The representation
$V_L$ of $S_L$ is irreducible and corresponds to the Young diagram with two rows
and $n/2$ columns.  The result we need is the following:

\begin{proposition}
\label{prop:decomp}
Assume $k$ is a field of characteristic zero.  In the following table, each
$S_L$-module is multiplicity free.  The set of irreducibles it contains
corresponds to the given set of partitions.
\begin{center}
\rm
\begin{tabular}{c|c}
$S_L$-module & Set of partitions of $n$ \\[.5ex]
\hline \\[-2ex]
$\Sym^2(V_L)$ & at most four parts, all even \\[1ex]
$\bw{V_L}$ & exactly four parts, all odd \\[1ex]
$V_L^{\otimes 2}$ & union of previous two sets \\[1ex]
$I_L^{(2)}$ & exactly four parts, all even
\end{tabular}\end{center}
\end{proposition}

This proposition is proved in \cite{HMSV}.  However, we will only need this
result in the case $n=8$, where it can easily be checked by computer or even by
hand.

\section{The skew-invariant cubic}

\noindent
{\it Until the final section of the paper, we take $k$ to be a field of
characteristic zero.}

\vskip 1ex

In this section we prove the existence and uniqueness of the skew-invariant
cubic $s$ of Theorem~\ref{mainthm}, as well as establishing the
formula for it in terms of graphs.  We fix once and for all a set $L$ of
cardinality eight and write $V$, \etc, in place of $V_L$, \etc\;  We 
often
write $G$ in place of $S_L=\Aut(L)$.

\begin{proposition}
\label{prop-s}
The space of skew-invariants in $\Sym^3(V)$ is one-dimensional.  It is spanned
by the skew-average of the cube of any matching.
\end{proposition}

The proof of Proposition~\ref{prop-s} is elementary linear
algebra.  However, to state it correctly we need some preparation.  For
now we allow $L$ to be any finite set of even cardinality $n$.  We write
$\mc{M}_L$ for the set of directed matchings on $L$.  The symmetric group
$S_L$ acts transitively on $\mc{M}_L$.  The alternating group $A_L$ clearly
does not act transitively, and thus has exactly two orbits.  We fix a
bijection
\begin{displaymath}
\sgn:\mc{M}_L/A_L \to \{ \pm 1\}.
\end{displaymath}
Thus for each matching $\Gamma$ we have a sign $\sgn{\Gamma}$, and for
$\sigma \in S_L$ we have $\sgn(\sigma \Gamma)=\sgn(\sigma) \sgn(\Gamma)$.

Now let $W$ be a $k$-vector space of dimension $n$ and let $\eta$ be a
non-degenerate symplectic form on $W$.  For an ordered pair $e=(i, j)$ in
$L \times L$ we define $\eta_e \in (W^*)^{\otimes \{i, j\}}$ as we defined
$\omega_e$ in \S \ref{s-graph}.  For a matching $\Gamma$ on $L$ we define
$\eta_{\Gamma}$ as the product $\prod \eta_e$, taken over the edges of
$\Gamma$.  Finally, we define $\eta^{L/2} \in (W^*)^{\otimes L}$ by the
formula
\begin{displaymath}
\eta^{L/2}=\sum_{\Gamma \in \mc{M}_L} \sgn(\Gamma) \eta_{\Gamma}.
\end{displaymath}
Clearly $\eta^{L/2}$ is skew-invariant for the action of $S_L$.

\begin{lemma}
\label{threetimes}
The space of $S_L$-skew-invariants in $(W^*)^{\otimes L}$ is one-dimensional
and spanned by $\eta^{L/2}$.  The group $\SL(W)$ leaves $\eta^{L/2}$
invariant.
\end{lemma}

\begin{proof}
This is the usual method of building a volume form out of a symplectic
form.
\end{proof}

We now prove the proposition.  We return to our original notation.

\begin{proof}[Proof of Proposition~\ref{prop-s}]
Recall that $V=((U^*)^{\otimes L})^{\SL(U)}$.  We thus have
\begin{displaymath}
\Sym^3{V}
=(V^{\otimes 3})_{S_3}
=((((U^*)^{\otimes L})^{\SL(U)})^{\otimes 3})_{S_3}
=((((U^*)^{\otimes L})^{\otimes 3})^{\SL(U)^3})_{S_3}
=((((U^*)^{\otimes 3})^{\otimes L})^{\SL(U)^3})_{S_3}.
\end{displaymath}
Here the subscript $S_3$ denotes co-invariants by the the symmetric group
$S_3$.  (Symmetric powers are most naturally defined by taking
co-invariants, not invariants.)  Putting $W=U^{\otimes 3}$, we may write
this formula as
\begin{displaymath}
\Sym^3{V}=(((W^*)^{\otimes L})^{\SL(U)^3})_{S_3}.
\end{displaymath}
Note that $\SL(U)^3 \rtimes S_3$ acts on $W$ and the action used in the above
formula is the natural one.  Furthermore, the action of $G=S_L$ given by
permuting factors commutes with $\SL(U)^3 \rtimes S_3$.

The space $W$ has a natural symplectic form, namely $\eta=\omega^{\otimes 3}$.
The group $\SL(U)^3 \rtimes S_3$ clearly preserves this form, and thus the map
$\SL(U)^3 \rtimes S_3 \to \GL(W)$ lands in $\Sp(W) \subset \SL(W)$.  We
now appeal to the lemma (note $W$ is eight-dimensional and $L$ has
cardinality eight).  We see that $\eta^{L/2}$ is non-zero and spans the space
of skew-invariants in $(W^*)^{\otimes L}$.  Furthermore, the group
$\SL(U)^3 \rtimes S_3$ leaves $\eta^{L/2}$ invariant.  We have thus shown
that the space of skew-invariants in
\begin{displaymath}
(((W^*)^{\otimes L})^{\SL(U)^3})^{S_3}
\end{displaymath}
is one-dimensional and spanned by $\eta^{L/2}$.  Since we are not in
characteristic 2 or 3, the map
\begin{displaymath}
(((W^*)^{\otimes L})^{\SL(U)^3})^{S_3} \to
(((W^*)^{\otimes L})^{\SL(U)^3})_{S_3} = \Sym^3{V}
\end{displaymath}
is an isomorphism.  Thus the space of skew-invariants in $\Sym^3{V}$ is
one-dimensional and spanned by the image $s$ of $\eta^{L/2}$.

We now wish to express the skew-invariant $s$ in terms of graphs.  We have,
by definition,
\begin{displaymath}
\eta^{L/2}=\sum_{\Gamma \in \mc{M}_L} \sgn(\Gamma) \eta_{\Gamma}.
\end{displaymath}
Now, $\eta_{\Gamma}$ is equal to $\prod \omega^{\otimes 3}_e$.  The image
of this in $\Sym^3{V}$ is just $\prod \omega^3_e$, which is, by definition,
$X_{\Gamma}^3$.  We thus have
\begin{displaymath}
s=\sum_{\Gamma \in \mc{M}_L} \sgn(\Gamma) X_{\Gamma}^3.
\end{displaymath}
(The skew-average of a cube of a matching
is equal to $\alpha s$ for some $\alpha \; \mid \; 8!$.)  This completes
the proof of the proposition.
\end{proof}

\section{The partial derivatives of $s$ span $I_8^{(2)}$}

In this section we prove that the 14 partial derivatives of $s$ span
$I_8^{(2)}$ and are linearly independent, thus establishing part of
Theorem~\ref{mainthm}.  We keep the notation from the previous section.

\begin{proposition}
\label{prop:quad}
Let $s$ be a non-zero skew-invariant element of $\Sym^3(V)$.  Then the 14
partial derivatives of $s$ are linearly independent and span $I^{(2)}$.
\end{proposition}

\begin{proof}
For an element $v^*$ of the dual space $V^*$ define a derivation
$\partial_{v^*}$ of $\Sym(V)$ by the formula
\begin{displaymath}
\partial_{v^*}(v_1 \cdots v_n)=\sum_{i=1}^n \langle v^*, v_i \rangle \cdot
(v_1 \cdots \hat{v}_i \cdots v_n)
\end{displaymath}
where the hat indicates that that factor is to be omitted.  We have a map
\begin{displaymath}
\Phi : V^* \otimes ks \to \Sym^2(V), \qquad v^* \otimes as \mapsto a
\partial_{v^*} s.
\end{displaymath}
The proposition states that $\Phi$ is injective with image $I^{(2)}$.

The crucial fact is that $\Phi$ is a map of $G$-modules.  Now, as a
$G$-module, $V$ is irreducible and corresponds to the Young diagram with
2 rows and 4 columns.  As with any representation of the symmetric group,
$V$ is self-dual.  Thus $V^* \otimes ks$ is the irreducible representation
with 4 rows and 2 columns (since $G$ acts on $ks$ by the sign representation).
Now, by Proposition~\ref{prop:decomp}, $\Sym^2(V)$ is multiplicity free.
Furthermore, that proposition shows that $I^{(2)}$ is irreducible and
corresponds to the Young diagram with 4 rows and 2 columns.  It thus follows
that $\Phi$ must have image contained in $I^{(2)}$.  Since the domain of
$\Phi$ is irreducible, it follows that $\Phi$ is either zero or injective.
But $\Phi$ cannot be zero since the non-zero polynomial $s$ must have some
non-zero partial derivative.  This proves the proposition.
\end{proof}

\section{The partial derivatives of $s$ have no linear syzygies --- set-up}

The goal of the next few sections is to establish the following proposition:

\begin{proposition}
\label{prop:nosyz}
The partial derivatives of $s$ have no linear syzygies.
\end{proposition}

This proposition means that if $\sum_{i=1}^{14} x_i \partial_i s=0$ with
$x_i$ in $\Sym^1(V)$ then $x_i=0$ for all $i$.  We will not prove
Proposition~\ref{prop:nosyz} in this section but we will reduce the proof to a
problem that we will soon solve.

Consider the composition
\begin{displaymath}
\wt{\Psi}:\End(V) \otimes \Sym^3(V) = V \otimes V^* \otimes \Sym^3(V) \to
V \otimes \Sym^2(V) \to \Sym^3(V)
\end{displaymath}
where the first map is the partial derivative map and the second map is
the multiplication map.  One easily verifies that $\wt{\Psi}$ is just the map
which expresses the action of the Lie algebra $\mf{gl}(V)=\End(V)$ on the
third symmetric power of its standard representation $V$.  We
are trying to show that $\wt{\Psi}$ induces an injection
\begin{displaymath}
\Psi:\End(V) \otimes ks \to I^{(3)}.
\end{displaymath}
(We know that $\Psi$ maps $\End(V) \otimes ks$ into $I^{(3)}$
since we know that the partial derivatives of $s$ belong to $I^{(2)}$.)
Indeed, the kernel of $\Psi$ is the space of linear syzygies between the
partial derivatives of $s$.  Now, the kernel of $\Psi$ is equal to $\mf{g}
\otimes ks$, where $\mf{g}$ is the annihilator in $\mf{gl}(V)$ of $s$.  Thus
Proposition~\ref{prop:nosyz} is equivalent to the following:

\begin{proposition}
\label{prop:l-zero}
We have $\mf{g}=0$.
\end{proposition}

We know two important things about $\mf{g}$:  first, $\mf{g}$ is a Lie
subalgebra of $\mf{gl}(V)$, as it is the annihilator of some element in a
representation of $\mf{gl}(V)$; and second, $\mf{g}$ is stable under the
group $G$, as the action map $\Psi$ is $G$-equivariant and $ks$ is stable
under $G$.  We will prove Proposition~\ref{prop:l-zero} by first classifying
the $G$-stable Lie subalgebras of $\mf{gl}(V)$ and then proving that $\mf{g}$
cannot be any of them except zero.

Before continuing, we note a few results:

\begin{proposition}
\label{prop:skew-ideal}
The skew-invariant cubic $s$ belongs to $I^{(3)}$.
\end{proposition}

\begin{proof}
We have already remarked that any element of the Lie algebra $\mf{gl}(V)$
takes $s$ into $I^{(3)}$.  Now, the identity matrix in $\mf{gl}(V)$ acts by
multiplication by 3 on $\Sym^3(V)$, and thus $3s$, and thus $s$, belongs to
$I^{(3)}$. An alternate argument, using the description of $s$ as a signed sum
of cubes of matchings (proof of Lemma \ref{threetimes}) was given in
\S \ref{other}.
\end{proof}

\begin{proposition}
The Lie algebra $\mf{g}$ is contained in $\mf{sl}(V)$.
\end{proposition}

\begin{proof}
The trace map $\mf{gl}(V) \to k$ is $G$-equivariant, where $G$ acts trivially
on the target.  Thus if $\mf{g}$ contained an element of non-zero trace it
would have to contain a copy of the trivial representation.  Thanks to
Proposition~\ref{prop:decomp}, we know that $\mf{gl}(V) \cong V^{\otimes 2}$
is multiplicity free as a representation of $G$.  Thus the one-dimensional
space spanned by the identity matrix is the only copy of the trivial
representation in $\mf{gl}(V)$.  Therefore, if $\mf{g}$ were not contained in
$\mf{sl}(V)$ then it would contain the center of $\mf{gl}(V)$.  However,
we know that the identity matrix does not annihilate $s$.  Thus $\mf{g}$ must
be contained in $\mf{sl}(V)$.
\end{proof}

\begin{proposition}
\label{prop:i2-i3}
Proposition~\ref{prop:nosyz} implies that $I^{(2)}$ generates $I^{(3)}$.
\end{proposition}

\begin{proof}
The image of $\Psi$ is exactly the subspace of $I^{(3)}$ generated by $I^{(2)}$.
Thus $I^{(2)}$ generates $I^{(3)}$ if and only if $\Psi$ is surjective.  Now, $V$
being 14 dimensional, the dimension of $\End(V)$ is 196.  It happens that this
is exactly the dimension of $I^{(3)}$ as well.  Thus the domain and target of
$\Psi$ have the same dimension, and so surjectivity is equivalent to injectivity.
\end{proof}

As remarked in the introduction one can prove Theorem~\ref{mainthm} by using
Proposition~\ref{prop:i2-i3} and a result from \cite{HMSV} which states that $I_n$
is generated in degrees two and three for all $n$.  We will not take this route,
however, and instead give an alternate proof in \S \ref{s:fin} that $I^{(2)}$
generates $I$.

\section{Interlude:  $G$-stable Lie subalgebras of $\mf{sl}(V)$}
\label{s-liesub}

In this section $G$ will denote an arbitrary finite group and $V$ an
irreducible representation of $G$ over an algebraically closed field $k$ of
characteristic zero.  We investigate the following general problem:

\begin{problem}
Determine the $G$-stable Lie subalgebras of $\mf{sl}(V)$.
\end{problem}

We do not obtain a complete answer to this question, but we do prove
strong enough results to determine the answer in our specific situation.
We will use the term \emph{$G$-subalgebra} to mean a $G$-stable Lie
subalgebra.

\subsection{Some structure theory}

Our first result is the following:

\begin{proposition}
\label{prop:solv}
Let $V$ be an irreducible representation of $G$.  Then every solvable
$G$-subalgebra of $\mf{sl}(V)$ is abelian and consists solely of semi-simple
elements.
\end{proposition}

\begin{proof}
Let $\mf{g}$ be a solvable subalgebra of $\mf{sl}(V)$.  By Lie's theorem,
$\mf{g}$ preserves a complete flag $0=V_0 \subset \cdots \subset V_n=V$.
The action of $\mf{g}$ on each one-dimensional space $V_i/V_{i-1}$ must factor
through $\mf{g}/[\mf{g}, \mf{g}]$; thus $[\mf{g}, \mf{g}]$ acts by zero on
$V_i/V_{i-1}$ and so carries $V_i$ into $V_{i-1}$.  The space $[\mf{g},
\mf{g}] V$ is therefore not all of $V$.  On the other hand, $[\mf{g}, \mf{g}]$
is $G$-stable and therefore so is $[\mf{g}, \mf{g}] V$.  From the
irreducibility of $V$ we conclude $[\mf{g}, \mf{g}] V=0$, from which it
follows that $[\mf{g}, \mf{g}]=0$.  Thus $\mf{g}$ is abelian.

Now let $R$ be the subalgebra of $\End(V)$ generated (under the usual
multiplication) by $\mf{g}$.  Let $R_s$ (resp.\ $R_n$) denote the set of
semi-simple (resp.\ nilpotent) elements of $R$.  Then $R_s$ is a subring of
$R$, $R_n$ is an ideal of $R$ and $R=R_s \oplus R_n$.  As $R_n^m=0$ for some
$m$, the space $R_n V$ is not all of $V$.  As it is $G$-stable it must be zero,
and so $R_n=0$.  We thus find that $R=R_s$ and so all elements of $R$, and
thus all elements of $\mf{g}$, are semi-simple.
\end{proof}

Let $V$ be a representation of $G$.  We say that $V$ is \emph{imprimitive}
if there is a decomposition $V=\bigoplus_{i \in I} V_i$ of $V$ into non-zero
subspaces, at least two in number, such that each element of $G$ carries each
$V_i$ into some $V_j$.  We say that $V$ is \emph{primitive} if it is not
imprimitive.  Note that primitive implies irreducible.  An irreducible
representation is imprimitive if and only if it is induced from a proper
subgroup.

\begin{proposition}
\label{prop:prim}
Let $V$ be an irreducible representation of $G$.  Then $V$ is primitive if
and only if the only abelian $G$-subalgebra of $\mf{sl}(V)$ is zero.
\end{proposition}

\begin{proof}
Let $V$ be an irreducible representation of $G$ and let $\mf{g}$ be a non-zero
abelian $G$-subalgebra of $\mf{sl}(V)$.  We will show that $V$ is
imprimitive.  By Proposition~\ref{prop:solv} all elements of $\mf{g}$ are
semi-simple.  We thus get a decomposition $V=\bigoplus V_{\lambda}$ of $V$ into
eigenspaces of $\mf{g}$ (each $\lambda$ is a linear map $\mf{g} \to k$).  As
$\mf{g}$ is $G$-stable, each element of $G$ must carry each $V_{\lambda}$ into
some $V_{\lambda'}$.  Note that if $V=V_{\lambda}$ for some $\lambda$
then $\mf{g}$ would consist of scalar matrices, which is impossible as $\mf{g}$
is contained in $\mf{sl}(V)$.  Thus there must be at least two non-zero
$V_{\lambda}$ and so $V$ is imprimitive.

We now establish the other direction.  Thus let $V$ be an irreducible
imprimitive representation of $G$.  We construct a non-zero abelian
$G$-subalgebra of $\mf{sl}(V)$.  Write $V=\bigoplus V_i$ where the elements of
$G$ permute the $V_i$.  Let $p_i$ be the endomorphism of $V$ given by
projecting onto $V_i$ and then including back into $V$ and let $\mf{g}$ be the
subspace of $\mf{gl}(V)$ spanned by the $p_i$.  Then $\mf{g}$ is an abelian
subalgebra of $\mf{gl}(V)$ since $p_i p_j=0$ for $i \ne j$.  Furthermore,
$\mf{g}$ is $G$-stable since for each $i$ we have $g p_i g^{-1}=p_j$ for some
$j$.  Intersecting $\mf{g}$ with $\mf{sl}(V)$ gives a non-zero abelian
$G$-subalgebra of $\mf{sl}(V)$ (the intersection is non-zero because $\mf{g}$
has dimension at least two and $\mf{sl}(V)$ has codimension one).
\end{proof}

We have the following important consequence of Proposition~\ref{prop:prim}:

\begin{corollary}
\label{cor:prim-ss}
Let $V$ be a primitive representation of $G$.  Then every $G$-subalgebra of
$\mf{sl}(V)$ is semi-simple.
\end{corollary}

\begin{proof}
Let $\mf{g}$ be a $G$-subalgebra of $\mf{sl}(V)$.  The radical of
$\mf{g}$ is then a solvable $G$-subalgebra and therefore vanishes.
Thus $\mf{g}$ is semi-simple.
\end{proof}

Proposition~\ref{prop:prim} can also be used to give a criterion for
primitivity.

\begin{corollary}
\label{cor:prim-test}
Let $V$ be an irreducible representation of $G$ such that each non-zero
$G$-submodule of $\mf{sl}(V)$ has dimension at least that of $V$.  Then $V$
is primitive.
\end{corollary}

\begin{proof}
Let $\mf{g}$ be a abelian $G$-subalgebra of $\mf{sl}(V)$.  We will show
that $\mf{g}$ is zero.  By Proposition~\ref{prop:solv} $\mf{g}$ consists of
semi-simple elements and is therefore contained in some Cartan subalgebra of
$\mf{sl}(V)$.  This shows that $\dim{\mf{g}} < \dim{V}$.  Thus, by our
hypothesis, $\mf{g}=0$.
\end{proof}

Let $V$ be a primitive $G$-module and let $\mf{g}$ be a $G$-subalgebra.
As $\mf{g}$ is semi-simple it decomposes as $\mf{g}=\bigoplus \mf{g}_i$ where
each $\mf{g}_i$ is a simple Lie algebra.  The $\mf{g}_i$ are called the
\emph{simple factors} of $\mf{g}$ and are unique.  As the simple factors are
unique, $G$ must permute them.  We call $\mf{g}$ \emph{prime} if the action of
$G$ on its simple factors is transitive.  Note that in this case the
$\mf{g}_i$'s are isomorphic and so $\mf{g}$ is ``isotypic.''  Clearly, every
$G$-subalgebra of $\mf{sl}(V)$ breaks up into a sum of prime subalgebras and
so it suffices to understand these.

\subsection{The action of a $G$-subalgebra on $V$}

We now consider how a $G$-stable subalgebra acts on $V$:

\begin{proposition}
\label{prop:iso}
Let $V$ be a primitive $G$-module, let $\mf{g}$ be a $G$-subalgebra of
$\mf{sl}(V)$ and let $\mf{g}=\bigoplus_{i \in I} \mf{g}_i$ be the decomposition
of $\mf{g}$ into simple factors.
\begin{enumerate}
\item The representation of $\mf{g}$ on $V$ is isotypic, that is, it is of the
form $V_0^{\oplus m}$ for some irreducible $\mf{g}$-module $V_0$.
\item We have a decomposition $V_0=\bigotimes_{i \in I} W_i$ where each $W_i$
is a faithful irreducible representation of $\mf{g}_i$.
\item We have $V_0 \cong V_0^g$ for each element $g$ of $G$.  (Here $V_0^g$
denotes the $\mf{g}$-module obtained by twisting $V_0$ by the automorphism
$g$ induces on $\mf{g}$.)
\item If $\mf{g}$ is a prime subalgebra then for any $i$ and $j$ one can
choose an isomorphism $f:\mf{g}_i \to \mf{g}_j$ so that $W_i$ and $f^* W_j$
become isomorphic as $\mf{g}_i$-modules.
\end{enumerate}
\end{proposition}

\begin{proof}
(1) Since $\mf{g}$ is semi-simple we get a decomposition $V=\bigoplus
V_i^{\oplus m_i}$ of $V$ as a $\mf{g}$-module, where the $V_i$ are pairwise
non-isomorphic simple $\mf{g}$-modules.  Each element $g$ of $G$ must take
each isotypic piece $V_i^{\oplus m_i}$ to some other isotypic piece
$V_j^{\oplus m_j}$ since the map $g:V \to V^g$ is $\mf{g}$-equivariant.  As
$V$ is primitive for $G$, we conclude that it must be isotypic for $\mf{g}$,
and so we may write $V=V_0^{\oplus m}$ for some irreducible $\mf{g}$-module
$V_0$.

(2) As $V_0$ is irreducible, it necessarily decomposes as a tensor product
$V_0=\bigotimes_{i \in I} W_i$ where each $W_i$ is an irreducible
$\mf{g}_i$-module.  Since the representation of $\mf{g}$ on $V=V_0^{\oplus m}$
is faithful so too must be the representation of $\mf{g}$ on $V_0$.  From this,
we conclude that each $W_i$ must be a faithful representation of $\mf{g}_i$.

(3) For any $g \in G$ the map $g:V \to V^g$ is an isomorphism of
$\mf{g}$-modules and so $V_0^{\oplus m}$ is isomorphic to $(V_0^{\oplus m})^g
=(V_0^g)^{\oplus m}$, from which it follows that $V_0$ is isomorphic to
$V_0^g$.

(4) Since $G$ acts transitively on the simple factors, given $i$ and $j$ we
can pick $g \in G$ such that $g \mf{g}_i= \mf{g}_j$.  The isomorphism of $V_0$
with $V_0^g$ then gives the isomorphism of $W_i$ and $W_j$ as
$\mf{g}_i$-modules.
\end{proof}

This proposition gives a strong numerical constraint on prime
subalgebras:

\begin{corollary}
\label{cor:prime-constraint}
Let $V$ be a primitive representation of $G$ and let $\mf{g}=\mf{g}_0^n$ be a
prime subalgebra of $\mf{sl}(V)$, where $\mf{g}_0$ is a simple Lie algebra.
Then $\dim{V}$ is divisible by $d^n$ where $d$ is the dimension of some
faithful representation of $\mf{g}_0$.  In particular, $\dim{V} \ge d_0^n$
where $d_0$ is the minimal dimension of a faithful representation of
$\mf{g}_0$.
\end{corollary}

\subsection{Self-dual representations}

Let $V$ be an irreducible self-dual $G$-module.  Thus we have a non-degenerate
$G$-invariant form $\langle, \rangle:V \otimes V \to k$.  Such a form is
unique up to scaling, and either symmetric or anti-symmetric.  We accordingly
call $V$ \emph{orthogonal} or \emph{symplectic}.

Let $A$ be an endomorphism of $V$.  We define the \emph{transpose} of $A$,
denoted $A^t$, by the formula
\begin{displaymath}
\langle A^t v, u \rangle=\langle v, Au \rangle.
\end{displaymath}
It is easily verified that $(AB)^t=B^t A^t$ and $({}^g A)^t={}^g (A^t)$.
We call an endomorphism $A$ \emph{symmetric} if $A=A^t$ and
\emph{anti-symmetric} if $A=-A^t$.  One easily verifies that the commutator
of two anti-symmetric endomorphisms is again anti-symmetric.  Thus the set of
all anti-symmetric endomorphisms forms a $G$-subalgebra of $\mf{sl}(V)$ which
we denote by $\mf{sl}(V)^-$.  In the orthogonal case $\mf{sl}(V)^-$ is
isomorphic to $\mf{so}(V)$ as a Lie algebra and $\bw{V}$ as a $G$-module,
while in the symplectic case it is isomorphic to $\mf{sp}(V)$ as a Lie algebra
and $\Sym^2(V)$ as a $G$-module.  We let $\mf{sl}(V)^+$ denote the space of
symmetric endomorphisms.

\begin{proposition}
\label{prop:self-dual}
Let $V$ be an irreducible self-dual $G$-module.  Assume that:
\begin{itemize}
\item $\Sym^2(V)$ and $\bw{V}$ have no isomorphic $G$-submodules; and
\item $\mf{sl}(V)^-$ has no proper non-zero $G$-subalgebras.
\end{itemize}
Then any proper non-zero $G$-subalgebra of $\mf{sl}(V)$ other than
$\mf{sl}(V)^-$ is commutative.  In particular, if $V$ is primitive then the
$G$-subalgebras of $\mf{sl}(V)$ are exactly 0, $\mf{sl}(V)^-$ and
$\mf{sl}(V)$.
\end{proposition}

\begin{proof}
Let $\mf{g}$ be a non-zero $G$-subalgebra of $\mf{sl}(V)$.  The intersection of
$\mf{g}$ with $\mf{sl}(V)^-$ is a $G$-subalgebra of $\mf{sl}(V)^-$ and
therefore either 0 or all of $\mf{sl}(V)^-$.  First assume that the
intersection is zero.  Since the spaces of symmetric and
anti-symmetric elements of $\mf{sl}(V)$ have no isomorphic $G$-submodules,
it follows that $\mf{g}$ is contained in the space of symmetric elements of
$\mf{sl}(V)$.  However, two symmetric elements bracket to an anti-symmetric
element.  It thus follows that all brackets in $\mf{g}$ vanish and so $\mf{g}$
is commutative.  Now assume that $\mf{g}$ contains all of $\mf{sl}(V)^-$.  It
is then a standard fact that $\mf{sl}(V)^-$ is a maximal subalgebra of
$\mf{sl}(V)$ and so $\mf{g}$ is either $\mf{sl}(V)^-$ or $\mf{sl}(V)$.  (To
see this, note that $\mf{sl}(V)=\mf{sl}(V)^- \oplus \mf{sl}(V)^+$ and so
to prove the maximality of $\mf{sl}(V)^-$ it suffices to show that
$\mf{sl}(V)^+$ is an irreducible representation of $\mf{sl}(V)^-$.  In
the orthogonal case this amounts to the fact that, as a representation of
$\mf{so}(V)$, the space $\Sym^2(V)/W$ is irreducible, where $W$ is the line
spanned by the orthogonal form on $V$.  The symplectic case is similar.)
\end{proof}

\section{The partial derivatives of $s$ have no linear syzygies --- completion
of proof}

We now complete the proof of Proposition~\ref{prop:nosyz}.  We return to our
previous notation.  We begin with the following:

\begin{proposition}
\label{prop:subalg}
Assume $k$ is algebraically closed.  The $G$-subalgebras of $\mf{sl}(V)$ are
exactly 0, $\mf{so}(V)$ and $\mf{sl}(V)$.
\end{proposition}

\begin{proof}
We begin by noting that any irreducible representation of the symmetric
group is defined over the reals (in fact, the rationals) and is therefore
orthogonal self-dual.  Thus $\mf{so}(V)=\mf{sl}(V)^-$ makes sense as a
$G$-subalgebra.

For our particular representation $V$, Proposition~\ref{prop:decomp} shows
that $\Sym^2(V)$ has five irreducible submodules of dimensions 1, 14, 14, 20
and 56, while $\bw{V}$ has two irreducible submodules of dimensions 35 and 56.
Furthermore, none of these seven irreducibles are isomorphic.  As all
irreducible submodules of $\mf{sl}(V)$ have dimension at least that of $V$
(which in this case is 14), we see from Corollary~\ref{cor:prim-test} that
$V$ is primitive.  (Note that the one-dimensional representation occurring
in $\Sym^2(V)$ is the center of $\mf{gl}(V)$ and does not occur in
$\mf{sl}(V)$.)

As $V$ is primitive, multiplicity free and self-dual, we can apply
Proposition~\ref{prop:self-dual}.  This shows that to prove the present
proposition we need only show that $\mf{so}(V)$ has no proper non-zero
$G$-subalgebras.  Thus assume that $\mf{g}'$ is a proper non-zero
$G$-subalgebra of $\mf{so}(V)$.  As $\mf{so}(V)=\bw{V}$ has two
irreducible submodules we see that $\mf{g}'$ must be one of these two
irreducibles.  In particular, this shows that $\mf{g}'$ must be prime and so
therefore isotypic.  Now, by examining the list of all simple Lie algebras,
 we
see that there are exactly four isotypic Lie algebras of dimension either 35
or 56:
\begin{displaymath}
\mf{g}_2^4, \qquad \mf{so}(8)^2, \qquad \mf{sl}(3)^7, \qquad \mf{sl}(6).
\end{displaymath}
The minimal dimensions of faithful representations of $\mf{g}_2$, $\mf{so}(8)$
and $\mf{sl}(3)$ are 7, 8 and 3.  As $7^4$, $8^2$ and $3^7$ are all bigger
than $\dim{V}$, Corollary~\ref{cor:prime-constraint} rules out the first
three Lie algebras above.  (One can also rule out $\mf{g}_2^4$ and
$\mf{sl}(3)^7$ by noting that the alternating group $A_8$ does not act
non-trivially on them.) We rule out $\mf{sl}(6)$ by using
Proposition~\ref{prop:iso} and noting that $\mf{sl}(6)$ has no faithful 14
dimensional isotypic representation --- this is proved in Lemma~\ref{lem:sl6}
below.  (One can also rule out $\mf{sl}(6)$ by noting that $A_8$ does not
act on it.)  This shows that $\mf{g}'$ cannot exist, and proves the
proposition.
\end{proof}

\begin{lemma}
\label{lem:sl6}
The Lie algebra $\mf{sl}(6)$ has exactly two non-trivial irreducible
representations of dimension $\le 14$:  the standard representation and its
dual.  It has no 14-dimensional faithful isotypic representation.
\end{lemma}

\begin{proof}
For a dominant weight $\lambda$ let $V_{\lambda}$ denote the irreducible
representation with highest weight $\lambda$.  If $\lambda$ and $\lambda'$
are two dominant weights then a general fact valid for any semi-simple
Lie algebra states
\begin{displaymath}
\dim{V_{\lambda+\lambda'}} \ge \max(\dim{V_{\lambda}}, \dim{V_{\lambda'}}).
\end{displaymath}
(To see this, recall the Weyl dimension formula:
\begin{displaymath}
\dim{V_{\lambda}}=\prod_{\alpha^{\vee}>0}
\frac{\langle \lambda+\rho, \alpha^{\vee} \rangle}{\langle \rho,
\alpha^{\vee} \rangle},
\end{displaymath}
where $\rho$ is half the sum of the positive roots and the product is taken
over the positive co-roots $\alpha^{\vee}$.  Then note that $\langle
\lambda, \alpha^{\vee} \rangle$ is positive for any dominant weight $\lambda$
and any positive co-root $\alpha^{\vee}$.  Thus $\dim{V_{\lambda+\lambda'}} \ge
\dim{V_{\lambda}}$.)

Now, let $\varpi_1, \ldots, \varpi_5$ be the fundamental weights for
$\mf{sl}(6)$.  The representation $V_{\varpi_i}$ is just $\bigwedge^i V$,
where $V$ is the standard representation.  For $2 \le i \le 4$ the space
$V_{\varpi_i}$ has dimension $\ge 15$.  Furthermore, a simple calculation
shows that
\begin{displaymath}
\dim{V_{2 \varpi_1}}=21, \qquad \dim{V_{\varpi_1+\varpi_5}}=168, \qquad
\dim{V_{2 \varpi_5}}=21.
\end{displaymath}
(Note that $V_{2 \varpi_1}$ is $\Sym^2(V)$, while $V_{2 \varpi_5}$ is its
dual.  This shows why they are 21-dimensional.  To compute the dimension of
$V_{\varpi_1+\varpi_5}$ we use the formula for the dimension of the relevant
Schur functor, \cite[Ex.~6.4]{FH}.) Thus only $V_{\varpi_1}$
and $V_{\varpi_5}$ have dimension at most 14, and
they each have dimension 6.  Since 6 does not divide 14 we find that there are
no non-trivial 14-dimensional isotypic representations.
\end{proof}

\begin{remark}
We can prove Proposition~\ref{prop:subalg} whenever $L$ has cardinality at most 14.  Perhaps
it is true for all $L$.
\end{remark}

We now have the following:

\begin{proposition}
\label{prop:so-ann}
The only element of $\Sym^3(V)$ annihilated by $\mf{so}(V)$ is zero.
\end{proposition}

\begin{proof}
As mentioned, $V$ has a canonical non-degenerate symmetric inner product.
Pick an orthonormal basis $\{x_i\}$ of $V$ and let $\{x_i^*\}$ be the dual
basis of $V^*$.  We can think of $\Sym(V)$ as the polynomial ring in the
$x_i$.  The space $\mf{so}(V)$ is spanned by elements of the form
$E_{ij}=x_i \otimes x_j^* - x_j \otimes x_i^*$.   Recall that, for an
element $s$ of $\Sym(V)$, the element $x_i \otimes x_j^*$ of $\End(V)$
acts on $s$ by $x_i \partial_j s$, where $\partial_j$ denotes differentiation
with respect to $x_j$.  Thus we see that $s$ is annihilated by $E_{ij}$ if
and only if it satisfies the equation
\begin{equation}
\label{eq1}
x_i \partial_j s = x_j \partial_i s.
\end{equation}
Therefore $s$ is annihilated by all of $\mf{so}(V)$ if and only if the above
equation holds for all $i$ and $j$.

Let $s$ be an element of $\Sym^3(V)$.  We now consider \eqref{eq1} for a fixed
$i$ and $j$.  Write
\begin{displaymath}
s=g_3(x_j)+g_2(x_j) x_i+g_1(x_j) x_i^2+g_0(x_j) x_i^3
\end{displaymath}
where each $g_i$ is a polynomial in $x_j$ whose coefficients are polynomials
in the $x_k$ with $k \ne i, j$.  Note that $g_0$ must be a constant by
degree considerations.  We have
\begin{displaymath}
\begin{split}
x_i \partial_j s &=  g_3'(x_j) x_i + g_2'(x_j) x_i^2 + g_1'(x_j) x_i^3 \\
x_j \partial_i s &= x_j g_2(x_j)+2 x_j g_1(x_j) x_i+3 x_j g_0(x_j) x_i^2.
\end{split}
\end{displaymath}
We thus find
\begin{displaymath}
g_2=0, \quad 2x_j g_1=g_3', \quad 3 x_j g_0=g_2', \quad g_1'=0.
\end{displaymath}
{}From this we deduce that $g_0=g_2=0$ and that $g_1$ is determined from
$g_3$.  The constraint on $g_3$ is that it must satisfy
\begin{equation}
\label{eq2}
g_3'(x_j)=x_j g_3''(x_j).
\end{equation}
Putting
\begin{displaymath}
g_3(x_j)=a+b x_j+c x_j^2+d x_j^3
\end{displaymath}
we see that \eqref{eq2} is equivalent to $b=d=0$.  We thus have
\begin{displaymath}
g_3(x_j)=a+c x_j^2, \qquad \textrm{and} \qquad g_1(x_j)=c
\end{displaymath}
and so
\begin{displaymath}
s = a + c (x_i^2+ x_j^2)
\end{displaymath}
is the general solution to \eqref{eq1}.

We thus see that if $s$ satisfies \eqref{eq1} for a particular $i$ and $j$
then $x_i$ and $x_j$ occur in $s$ with only even powers.  Thus if $s$
satisfies \eqref{eq1} for all $i$ and $j$ then all variables appear to an
even power.  This is impossible, unless $s=0$, since $s$ has degree three.
Thus we see that zero is the only solution to \eqref{eq1} which holds for
all $i$ and $j$.
\end{proof}

\begin{remark}
The above computational proof can be made more conceptual.  By considering
the equation \eqref{eq1} for a fixed $i$ and $j$ we are considering the
invariants of $\Sym^3(V)$ under a certain copy of $\mf{so}(2)$ sitting inside
of $\mf{so}(V)$.  The representation $V$ restricted to $\mf{so}(2)$ decomposes
as $V' \oplus T$ where $V'$ is the standard representation of $\mf{so}(2)$
and $T$ is a 12-dimensional trivial representation of $\mf{so}(2)$.  We
then have
\begin{displaymath}
\Sym^3(V)^{\mf{so}(2)}=\bigoplus_{i=0}^3 \Sym^i(V')^{\mf{so}(2)} \otimes
\Sym^{3-i}(T).
\end{displaymath}
Finally, our general solution to \eqref{eq1} amounts to the fact that the
ring of invariant $\Sym(V')^{\mf{so}(2)}$ is generated by the norm form
$x_i^2+x_j^2$.
\end{remark}

We can now prove Proposition~\ref{prop:l-zero}, which will establish
Proposition~\ref{prop:nosyz}.

\begin{proof}[Proof of Proposition~\ref{prop:l-zero}]
To prove $\mf{g}=0$ we may pass to the algebraic closure of $k$; we thus
assume $k$ is algebraically closed.  By Proposition~\ref{prop:subalg}, the Lie
algebra $\mf{g}$ must be 0, $\mf{so}(V)$ or $\mf{sl}(V)$.  By
Proposition~\ref{prop:so-ann} $\mf{g}$ cannot be $\mf{so}(V)$ or $\mf{sl}(V)$
since it annihilates $s$ and $s$ is non-zero.  Thus $\mf{g}=0$.
\end{proof}

\section{Completion of the proof of Theorem~\ref{mainthm}}
\label{s:fin}

We now complete the proof of Theorem~\ref{mainthm}.  Before doing so, we need to
review some commutative algebra.  In this section we work over an algebraically
closed field $k$ of characteristic zero.

\subsection{Betti numbers of modules over polynomial rings}

Let $P$ be a graded polynomial ring over $k$ in finitely many indeterminates,
each of positive degree.  Let $M$ be a finite $P$-module.  One can then find a
surjection $F \to M$ with $F$ a finite free module having the following
property: if $F' \to M$ is another surjection from a finite free module then
there is a surjection $F' \to F$ making the obvious diagram commute.  We call
$F \to M$ a \emph{free envelope} of $M$.  It is unique up to non-unique
isomorphism.  As an example, if $M$ is generated by its degree $d$ piece then
we can take $F$ to be $P[-d] \otimes M^{(d)}$ where the tensor product is over
$k$ and $P[-d]$ is the free $P$-module with one generator in degree $d$.

Let $M$ be a finite free $P$-module.  We can build a resolution of $M$ by using
free envelopes:
\begin{displaymath}
\cdots \to F_2 \to F_1 \to F_0 \to M \to 0
\end{displaymath}
Here $F_0$ is the free envelope of $M$ and $F_{i+1}$ is the free envelope of
$\ker(F_i \to F_{i-1})$.  Define integers $b_{i, j}$ by
\begin{displaymath}
F_i=\bigoplus_{j \in \Z} P[-i-j]^{\oplus b_{i, j}}.
\end{displaymath}
These integers are called the \emph{Betti numbers} of $M$ and the collection of
them all the \emph{Betti diagram} of $M$.  They are independent of the choice of
free envelopes, as $b_{i, j}$ is also the dimension of the $j$th graded piece of
$\Tor_i^P(M, P/I)$, where $I$ is ideal of positive degree elements.  The Betti
numbers have the following properties:
\begin{itemize}
\item[(B1)] We have $b_{i, j}=0$ for all but finitely many $i$ and $j$.  This is
because each $F_i$ is finitely generated and $F_i=0$ for $i$ large by Hilbert's
theorem on syzygies.
\item[(B2)] We have $b_{i, j}=0$ for $i<0$.  This follows from the definition.
\item[(B3)] If $b_{i_0, j}=0$ for $j \le j_0$ then $b_{i, j}=0$ for all $i
\ge i_0$ and $j \le j_0$.  This follows from the fact that if $d$ is the lowest
degree occurring in a module $M$ and $F \to M$ is a free envelope then $F^{(d)}
\to M^{(d)}$ is an isomorphism, and thus the lowest degree occurring in
$\ker(F \to M)$ is $d+1$.
\item[(B4)] In particular, if $M$ is in non-negative degrees then $b_{i, j}=0$
for $j<0$.
\item[(B5)] Let $f(k)=\dim{M^{(k)}}$ (resp.\ $g(k)=\dim{P^{(k)}}$) denote the
Hilbert function of $M$ (resp.\ $P$). Then
\begin{displaymath}
f(k)=\sum_{i, j \in \Z} (-1)^i \cdot b_{i, j} \cdot g(k-i-j).
\end{displaymath}
This follows by taking the Euler characteristic of the $k$th graded piece of
$F_{\bullet} \to M$.
\end{itemize}
In particular we see that if $M$ is in non-negative degrees then its Betti
diagram is contained in a bounded subset of the first quadrant.

\subsection{Betti numbers of graded algebras}

Let $R$ be a finitely generated graded $k$-algebra, which we assume for
simplicity to be generated by its degree one pice.  We let $P=\Sym(R^{(1)})$
be the graded polynomial algbera on the first graded piece.  We have a
natural surjective map $P \to R$ and so $R$ is a $P$-module.  We can thus speak
of the Betti numbers of $R$ as a $P$-module.  We call these the Betti numbers
of $R$.

Assume now that the ring $R$ is Gorenstein and a domain.  The canonical module
$\omega_R$ of $R$ is then naturally a graded module.  Furthermore, there exists
an integer $a$, called the \emph{$a$-invariant} of $R$, such that $\omega_R$ is
isomorphic to $R[a]$.  We now have the following important property of the
Betti numbers of $R$:
\begin{itemize}
\item[(B6)] We have $b_{i, j}=b_{r-i, d+a-j}$ where $d=\dim{R}$ is the Krull
dimension of $R$, $r=\dim{P}-\dim{R}$ is the codimension of $\Spec(R)$ in
$\Spec(P)$ and $a$ is the $a$-invariant of $R$.
\end{itemize}
No doubt this formula appears in the literature, but we will derive it here for completeness.   
We have $\mathrm{Ext}^i_P(R, \omega_P) \cong \omega_R$ if 
$i = r$ and $0$ if $i \neq r$.   If $n$ is the dimension of $P$, then $\omega_P \cong P[-n]$.  Since $R$ is Gorenstein we have 
$\omega_R \cong R[a]$.
Therefore we obtain a minimal free resolution $G_\bullet$ of $R[a]$ 
by $G_i$  = $\mathrm{Hom}_P(F_{r-i}, P[-n])$.
We have $G_\bullet[-a]$ is a minimal free resolution of $R$, and by uniqueness of the resolution we therefore have $G_i[-a] \cong F_i$ for each $i$.  Now $G_i[-a] \cong \oplus_{j'} P[-n-r+i+j'-a]$, and so 
$$\oplus_{j'} P[-n+r-i+j'-a]^{b_{r-i,j'}} \cong \oplus_j P[-i-j]^{b_{i,j}}.$$
Equating components of the same degree gives $-n+r-i+j'-a = -i-j$, or 
$j' = n-r + a - j$.  Hence $b_{i,j} = b_{r-i, n-r+a-j} = b_{r-i, d+a-j}$.

\subsection{Completion of the proof of Theorem~\ref{mainthm}}
\label{s:betti}

We now return to our previous notation.  Thus $L$ is a fixed eight element set,
$R=R_L$, $k$ is a field of characteristic zero, \etc\ We begin with the
following:

\begin{proposition}
\label{prop:goren}
The ring $R$ is Gorenstein with $a$-invariant $-2$.
\end{proposition}

\begin{proof}
We first recall a theorem of Hochster-Roberts \cite[Theorem~6.5.1]{BrunsHerzog}:
if $V$ is a representation of the reductive group $G$ (over a field of
characteristic zero) then the ring of invariants $(\Sym{V})^G$ is
Cohen-Macaulay.  As our ring $R$ can be realized in this manner, with $V$ being
the space of $2 \times 8$ matrices and $G=\SL(2) \times T$, where $T$ is the
maximal torus in $\SL(8)$, we see that $R$ is Cohen-Macaulay.  We now recall
a theorem of Stanley \cite[Corollary~4.4.6]{BrunsHerzog}:  if $R$ is a
Cohen-Macaulay ring generated in degree one with Hilbert series $f(t)/(1-t)^d$,
where $d$ is the Krull dimension of $R$, then $R$ is Gorenstein if and only
if the polynomial $f$ is symmetric.  Furthermore, if $f$ is symmetric then
the $a$-invariant of $R$ is given by $\deg{f}-d$.  Going back to our situation,
the Hilbert series of our ring was given in \S \ref{s:1-1}.  The numerator
is symmetric of degree four and the denominator has degree six.  We thus see
that $R$ is Gorenstein with $a=-2$.
\end{proof}

We can now deduce the Betti diagram of $R$:

\begin{proposition}
\label{prop:betti}
The Betti diagram of $R$ is given by:
\vskip 2ex
\begin{center} \rm
\begin{tabular}{|c||c|c|c|c|c|c|c|c|c|}
\hline
& \hskip .2em 0 \hskip .2em & 1 & 2 & 3 & 4 & 5 & 6 & 7 &
\hskip .2em 8 \hskip .2em \\
\hline
\hline
0 & 1 & 0 & 0 & 0 & 0 & 0 & 0 & 0 & 0 \\
\hline
1 & 0 & 14 & 0 & 0 & 0 & 0 & 0 & 0 & 0 \\
\hline
2 & 0 & 0 & 175 & 512 & 700 & 512 & 175 & 0 & 0 \\
\hline
3 & 0 & 0 & 0 & 0 & 0 & 0 & 0 & 14 & 0 \\
\hline
4 & 0 & 0 & 0 & 0 & 0 & 0 & 0 & 0 & 1 \\
\hline
\end{tabular}
\end{center}
\vskip 2ex
The $i$-axis is horizontal and the $j$-axis vertical.
All $b_{i, j}$ outside of the above range are zero.
\end{proposition}

\begin{proof}
We first note that (B6) gives $b_{8-i, 4-j}=b_{i, j}$ as $r=8$, $d=6$ and
$a=-2$ in our situation.  We thus have the symmetry of the table.
Now, by (B2) and (B4) we have $b_{i, j}=0$ if either $i$ or $j$ is negative.
We thus see that $b_{i, j}=0$ if $i>8$ or $j>4$ by symmetry.  Next, observe
that $P \to R$ is the free envelope of $R$, where $P=\Sym(V)$.  This gives the
$i=0$ column of the table.  We now look at the $i=1$ column.  We know that the
14 generators have no linear relations and so $b_{1, 0}=0$.  By (B3) we have
$b_{i, 0} =0$ for $i \ge 1$.  We also know that there are 14 quadric relations
and so $b_{1, 1}=14$.  We now look at the $i=2$ column of the table.  We have
proved (Proposition~\ref{prop:nosyz}) that the 14 quadric relations have no
linear syzygies; this gives $b_{2, 1}=0$.  Using (B3) again, we conclude
$b_{i, 1}=0$ for $i \ge 2$.  We have thus completed the first two rows of the
table.  The last two rows can then be completed by symmetry.  The middle row
can now be determined from (B5) by evaluating both sides at $k=2, \ldots,
10$ and solving the resulting upper triangular system of equations for
$b_{i ,2}$.  (In fact, the computation is simpler than that since $b_{i, 2}=
b_{8-i, 2}$ and we know $b_{0, 2}=b_{1, 2}=0$, the latter vanishing
coming from Proposition~\ref{prop:i2-i3}.)
\end{proof}

Proposition~\ref{prop:betti} --- in particular, the $i=1$ column of the table ---
shows that $I_8$ is generated by its degree two piece.  Thus we have proved
Theorem~\ref{mainthm}.

\begin{remark}
The resolution of $R$ as a $P$-module, without any consideration of grading,
is given by Freitag and Salvati Manni \cite[Lemma~1.3, Theorem~ 1.5]{fs}.
It was obtained by computer.
\end{remark}

\section{Working over $\Z$:  Proof and discussion of  Theorem~\ref{mainthm2}}
\label{s-integrality}

\noindent
{\it In this section we take the base ring $k$ to be $\Z$.}

\vskip 1ex

We begin with a short discussion of linear algebra over $\Z$.  Let $M$ be a
finite free $\Z$-module and let $N$ be a submodule.  We say that $N$ is
\emph{saturated} (in $M$) if whenever $nx$ belongs to $N$, with $n \in \Z$
and $x \in M$, we have that $x$ belongs to $N$.  Of course, $N$ is saturated
if and only if it is a summand of $M$.  Note that if $N$ and $N'$ are
saturated submodules such that $N \otimes \Q=N' \otimes \Q$ then $N=N'$.
Finally, we remark that $I^{(n)}$, the $n$th graded piece of the ideal, is
a saturated submodule of $\Sym^n(V)$.  This is easily seen as $I^{(n)}$ is
the kernel of $\Sym^n(V) \to R^{(n)}$, and $R^{(n)}$ is torsion free.

We begin our discussion proper by giving an explicit formula for $s'$ in terms
of the basis of non-crossing matchings (see Figure~\ref{fig:noncross} for a
listing of these 14 generators):
\begin{displaymath}
\begin{split}
s' \; = \quad &x_1x_2(x_1+x_2) + x_1x_2(z_1+z_2+z_3+z_4+z_5+z_6+z_7+z_8)-
(x_1y_2y_4 + x_2y_3y_1) \\
&+(x_1z_2z_6 + x_2z_3z_7 + x_1z_4z_8 + x_2z_5z_1)
+(y_1z_2z_6+y_2z_3z_7+y_3z_4z_8+y_4z_5z_1) \\
&-(z_1z_2z_3 + z_2z_3z_4 + z_3z_4z_5 + z_4z_5z_6 + z_5z_6z_7 +z_6z_7z_8+
z_7z_8z_1+z_8z_1z_2).
\end{split}
\end{displaymath}
(Note that for the formula to be unambiguous we need to specify how the edges of
the matchings are directed.  Label the vertices from 1 to 8 going clockwise,
starting at any vertex.  Then the edges are directed to point from smaller to
larger numbers.  The choice of starting vertex does not affect the above
expression for $s'$.) This formula was found with the aid of a computer by taking
the skew-average of a particular element of $\Sym^3(V)$.  This element is related
to the generalized Segre cubics of \cite{HMSV} and was chosen because it has a
large isotropy subgroup.  The right side of the expression for $s'$ is visibly
non-zero as we are in the polynomial ring on the $x$, $y$ and $z$ variables.

\begin{figure}[ht]
\includegraphics[scale=0.6]{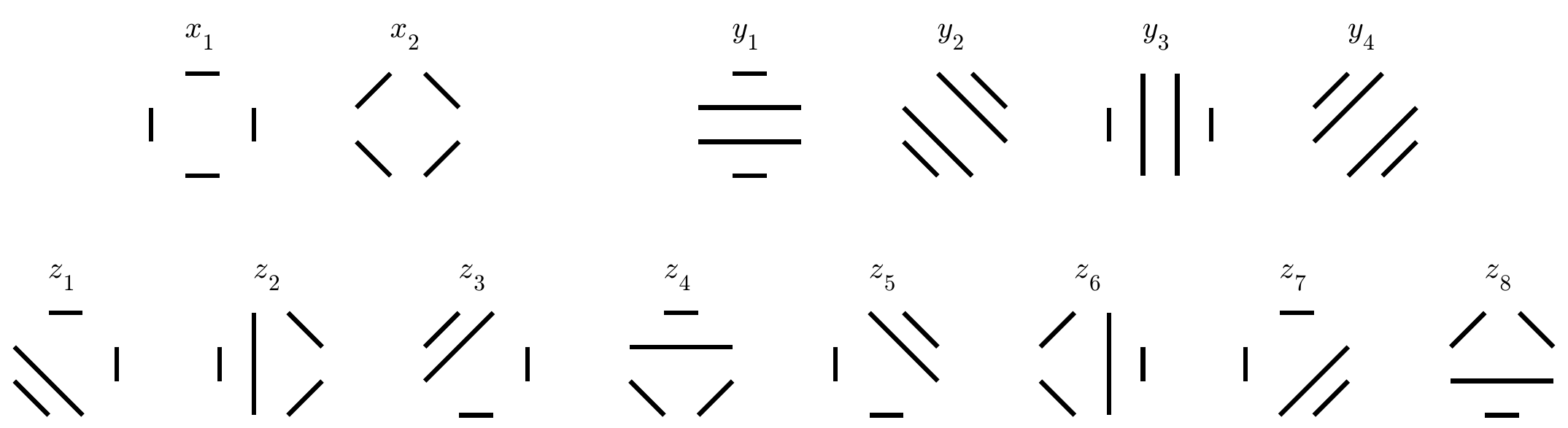}
\caption{The fourteen non-crossing matchings.}
\label{fig:noncross}
\end{figure}

\begin{proposition}
The 14 partial derivatives of $s'$ give a basis for $I^{(2)}$ as a $\Z$-module.
\end{proposition}

\begin{proof}
The reader may check that each of the 14 partial derivatives of $s'$ contains
a monomial with unit coefficient which does not appear in the other 13 partial
derivatives.  For example $\frac{\partial s'}{\partial x_1}$ contains the
monomial term $x_2^2$ with coefficient $+1$, and this monomial does
not appear in any of the other 13 partial derivatives.  It follows that the
$\Z$-module spanned by these 14 quadrics inside of $\Sym^2(V)$ is saturated.
As these 14 quadrics give a basis for $I^{(2)} \otimes \Q$ and $I^{(2)}$ is
saturated in $\Sym^2(V)$ we see that they must in fact give a basis for
$I^{(2)}$ as a $\Z$-module.  Thus the partial derivatives of $s'$ give a basis
for $I^{(2)}$ as a $\Z$-module.
\end{proof}

We can now prove Theorem~\ref{mainthm2}:

\begin{proof}[Proof of Theorem~\ref{mainthm2}]
Let $J$ be the ideal of $\Sym(V)$ generated by $s'$ and its 14 partial
derivatives.  We must show $I=J$.  By the main theorem of \cite{hmsv1} it suffices
to show $I^{(2)}=J^{(2)}$, $I^{(3)}=J^{(3)}$, and $I^{(4)} = J^{(4)}$.
The previous proposition
established the first of these equalities.  We must establish the second.

Now, as in the previous proof, we know that $J^{(3)} \otimes \Q=I^{(3)}
\otimes \Q$ and that $I^{(3)}$ is saturated in $\Sym^3(V)$.  Thus to prove
$I^{(3)}=J^{(3)}$ it suffices to show that $J^{(3)}$ is saturated.
Unfortunately we do not how to do this by hand without an exorbitant amount
of work.  However, we can use the computer algebra system Magma \cite{Magma}
to check this; see the website%
\footnote{http://www-personal.umich.edu/$\sim$howardbj/8points.html}
of the first author for the code.  We perform this check as follows.
Create a $197 \times 560$ matrix $M$, with the columns indexed by the cubic
monomials in the 14 non-crossing variables and with the rows corresponding to
the 196 possible products $u \frac{\partial s'}{\partial v}$, for $u, v
\in \{x_1, x_2, y_1, \ldots, y_4, z_1, \ldots, z_8\}$ as well as $s'$.  We
verify that the elementary divisors of $M$ are all equal to $1$ by computing
the Smith normal form of $M$ in Magma.  Similarly we check by computer 
that in degree 4 that the $197 \cdot 14=2758$ quartics generated by 
the cubics span a saturated sub-lattice of $\Sym^4(R^{(1)})$ of rank $1295
=\rk{I^{(4)}}$.  (We compute the rank of $I^{(4)}$ by expressing it as
$\rk(\Sym^4(R^{(1)}))-\rk(R^{(4)})$.)

Now suppose that $J_0$ is the ideal generated by the 14 partial derivatives
alone.  Similar to the above, we define a $196 \times 560$ matrix $M_0$, with
rows indexed by the $u \frac{\partial s'}{\partial v}$, and we find using Magma
that it has an elementary divisor equal to $3$.  This shows that $I \ne J_0$
and so the cubic $s'$ is necessary when 3 is not invertible.
\end{proof}

\subsection{The Betti diagram in characteristic $p > 0$}

We now explain how the results and proofs in \S \ref{s:betti} can be adapted to
work over a field $k$ of positive characteristic.  First, in \cite{HMSV} we will
show that $R$ is Cohen-Macaulay.  In fact, we will show that $R_n$, over any
field, is Cohen-Macaulay.  (One cannot use the Hochster-Roberts theorem to prove
this as the group $\SL(2) \times T$ does not have a semi-simple representation
category in positive characteristic.)  The proof of Proposition~\ref{prop:goren}
then carries over to show that $R$ is Gorenstein with $a=-2$.  Note that
Stanley's theorem is true over any field and that the Hilbert series of $(R)_k$
is independent of the field $k$ as $(R)_{\Z}$ is flat over $\Z$.

We thus see that Proposition~\ref{prop:goren} holds true over any field.  We
now turn to Proposition~\ref{prop:betti}.  We first note that the same
reasoning used in the proof of Proposition~\ref{prop:i2-i3} shows that
$b_{1, 2}=b_{2, 1}$ over any field.  When $\chr{k} \ne 3$ Theorem~\ref{mainthm2}
gives $b_{1, 2}=0$.  The proof of Proposition~\ref{prop:betti} then carries
over exactly the same to this situation.  Thus the Betti diagram is the same
as in characteristic zero.  Now consider the case where $\chr{k}=3$.
Theorem~\ref{mainthm2} then gives $b_{1, 2}=1$ and so $b_{2, 1}=1$ as well.
>From this, one may conclude $b_{3, 1}=0$ (as there must be at least two relations
to produce a syzygy) and thus that $b_{i, 1}=0$ for $i \ge 3$.  We thus have the
first two rows of the table and by symmetry get the last two rows.  One can again
compute the middle row using (B5).  However, all that goes into this computation
is the alternating sum of the $b_{i, j}$ along diagonals $i+j=n$.  Since
$b_{1, 2}$ still equals $b_{2, 1}$, the alternating sum along the $i+j=3$ diagonal
is still the same.  From this we see that the middle row is the same as in
characteristic zero.

To sum up, we have proved the following:

\begin{proposition}
Let $k$ be a field.  If $\chr{k} \ne 3$ then the Betti diagram of $(R)_k$ is
the same as that given in Proposition~\ref{prop:betti}.  If $\chr{k}=3$ then
the only change is that $b_{1, 2}=b_{2, 1}=1$ and, symmetrically, $b_{7, 2}=
b_{6, 3}=1$.
\end{proposition}

\bigskip

{\tiny

Benjamin Howard:
Department of Mathematics,
University of  Michigan,
Ann Arbor, MI 48109, USA,
howardbj@umich.edu

\smallskip

John Millson:
Department of Mathematics,
University of Maryland,
College Park, MD 20742, USA,
jjm@math.umd.edu

\smallskip

Andrew Snowden:
Department of Mathematics,
Princeton University,
Princeton, NJ 08544, USA,
asnowden@math.princeton.edu

\smallskip

Ravi Vakil:
Department of Mathematics,
Stanford University,
Stanford, CA 94305-2125, USA,
vakil@math.stanford.edu

}

\end{document}